\documentclass[12pt,twoside]{amsart}
\usepackage{amssymb,amsmath,amsthm, amscd, enumerate, mathrsfs}
\usepackage{graphicx, hhline}
\usepackage[all]{xy}
\usepackage[dvipdfmx]{hyperref}

\title{On Nakayama's theorem}
\author{Osamu Fujino}
\date{2021/7/16, version 0.11}
\dedicatory{Dedicated to Professor Noboru Nakayama on the 
occasion of his sixtieth birthday}
\address{Department of 
Mathematics, Graduate School of Science, 
Kyoto University, Kyoto 606-8502, Japan}
\email{fujino@math.kyoto-u.ac.jp}

\DeclareMathOperator{\Supp}{Supp}
\DeclareMathOperator{\codim}{codim}

\newtheorem{thm}{Theorem}[section]
\newtheorem{lem}[thm]{Lemma}

\newtheorem{prop}[thm]{Proposition}

\theoremstyle{definition}
\newtheorem{step}{Step}
\newtheorem{defn}[thm]{Definition}
\newtheorem{rem}[thm]{Remark}
\newtheorem*{ack}{Acknowledgments}  
\makeatletter
    
    \@addtoreset{equation}{section}
\makeatother

\begin{document}

\maketitle 

\begin{abstract} 
The main purpose of this paper is 
to make Nakayama's theorem more accessible. 
We give a proof of Nakayama's theorem based on 
the negative definiteness of intersection matrices of exceptional curves. 
In this paper, we treat Nakayama's theorem on algebraic 
varieties over any algebraically closed field of 
arbitrary 
characteristic although Nakayama's original statement 
is formulated for complex analytic 
spaces. 
\end{abstract}

%\tableofcontents

\section{Introduction}

In this paper, a {\em{variety}} means an integral separated scheme of 
finite type over an algebraically closed field $k$ of any characteristic. 
The following theorem is very well known and 
plays a crucial role 
in the theory of higher-dimensional minimal models. 

\begin{thm}\label{f-thm1.1}
Let $f\colon X\to Y$ be a projective birational morphism 
from a smooth surface $X$ to a normal surface $Y$. 
Then the intersection matrix of the $f$-exceptional 
curves 
is negative definite. 
\end{thm}

The main purpose of this paper is to make the 
following theorem by Noboru Nakayama more accessible. 
Here we treat varieties over any algebraically closed field $k$ of 
arbitrary 
characteristic although the original statement is formulated for complex analytic 
spaces. 

\begin{thm}[{Nakayama's theorem, see \cite[Chapter III, 5.10.~Lemma (3)]
{nakayama}}]\label{f-thm1.2}
Let $f\colon X\to Y$ be a projective 
surjective morphism 
from a smooth variety $X$ onto a normal variety $Y$. 
Let $D$ be an $\mathbb R$-divisor 
on $X$. Then there exists an effective $f$-exceptional divisor 
$E$ on $X$ such that 
$$
\left(f_*\mathcal O_X(\lfloor tD\rfloor) \right)^{**}= 
f_*\mathcal O_X(\lfloor t(D+E)\rfloor)
$$ 
holds for every positive real number $t$. 
\end{thm}

Theorem \ref{f-thm1.2} 
has already played a fundamental role 
in \cite{takayama}, \cite{paun-takayama}, 
\cite{campana-paun}, and so on. 
Nakayama's original proof in \cite{nakayama} 
uses his theory of {\em{relative $\sigma$-decompositions}} 
and {\em{relative $\nu$-decompositions}} developed in 
\cite[Chapter III, \S1, \S3, and \S4]{nakayama}. 
Hence we had to study $N_\sigma$ and $N_\nu$ to understand 
Theorem \ref{f-thm1.2}. 
Our argument in this paper clarifies that 
Theorem \ref{f-thm1.2} is an easy consequence of Theorem \ref{f-thm1.1}. 
Roughly speaking, Theorem \ref{f-thm1.2} is a 
variant of the negativity lemma (see, for example, 
\cite[Lemma 2.3.26]{fujino-foundations}). 
We do not need $N_\sigma$ and $N_\nu$. 

\begin{ack}
The author was partially 
supported by JSPS KAKENHI Grant Numbers 
JP16H03925, JP16H06337. 
\end{ack}

\section{Preliminaries}

Let us start with the definition of exceptional divisors. 

\begin{defn}[Exceptional divisors]\label{f-def2.1} 
Let $f\colon X\to Y$ be a proper 
surjective morphism between normal varieties. 
Let $E$ be a Weil divisor on $X$. 
We say that $E$ is {\em{$f$-exceptional}} 
if $\codim_Y\! f(\Supp E)\geq 2$. 
We note that $f$ is not always assumed to be birational. 
\end{defn}

In order to understand Theorem \ref{f-thm1.2}, 
we need the following definitions. 

\begin{defn}\label{f-def2.2}
Let $D=\sum _i a_i D_i$ be an $\mathbb R$-divisor on a normal 
variety $X$ such that $D_i$ is a prime divisor on $X$ for 
every $i$ and that $D_i\ne D_j$ for $i\ne j$. 
We put 
$$
D^+=\sum _{a_i>0} a_i D_i \quad \text{and} 
\quad D^-=-\sum _{a_i<0} a_i D_i \geq 0. 
$$ 
Note that 
$$
D=D^+-D^-
$$
obviously holds. 
For every real number $x$, $\lfloor x\rfloor$ is the integer 
defined by $x-1<\lfloor x\rfloor \leq x$. 
We put 
$$
\lfloor D\rfloor =\sum _i \lfloor a_i\rfloor D_i
$$ 
and call it the {\em{round-down}} of $D$. 
\end{defn}

\begin{defn}\label{f-def2.3}
Let $\mathcal F$ be a coherent sheaf on a normal variety $X$. 
We put 
$$
\mathcal F^*=\mathcal{H}om_{\mathcal O_X} (\mathcal F, \mathcal O_X) 
$$ and 
$$
\mathcal F^{**}=(\mathcal F^*)^*. 
$$ 
Then there exists a natural map 
$
\mathcal F\to \mathcal F^{**}$. 
If this map $\mathcal F\to \mathcal F^{**}$ is an isomorphism, 
then $\mathcal F$ is called a {\em{reflexive}} 
sheaf. 
\end{defn}

For the basic properties of reflexive 
sheaves, see \cite[Section 1]{hartshorne}. 
We prepare an easy lemma for the reader's 
convenience. 

\begin{lem}\label{f-lem2.4}
Let $V$ be a smooth surface and let $C_1,\ldots, C_m$ 
be effective Cartier divisors on $V$ such that 
the intersection matrix $\left(C_i\cdot C_j\right)$ is negative 
definite and that $C_i\cdot C_j\geq 0$ for 
$i\ne j$. 
Let $$D=B+\sum _{i=1}^ma_i C_i$$ be an $\mathbb R$-divisor on $V$. 
Assume that 
\begin{itemize}
\item[(a)] $D\cdot C_i\leq 0$ {\em{(}}resp.~$D\cdot C_i<0${\em{)}} 
for every $i$, and 
\item[(b)] $B\cdot C_i\geq 0$ for every $i$. 
\end{itemize} 
Then $a_i\geq 0$ {\em{(}}resp.~$a_i>0${\em{)}} for every $i$. 
\end{lem}

\begin{rem}\label{f-rem2.5}
In Lemma \ref{f-lem2.4}, 
$C_i$ may be reducible and disconnected. 
It may happen that 
$C_i$ and $C_j$ have some common irreducible components for $i\ne j$. 
\end{rem}

\begin{proof}[Proof of Lemma \ref{f-lem2.4}]
By (b), $B\cdot C_j\geq 0$ for every $j$. 
Hence $\left(\sum _{i=1}^m a_i C_i\right)\cdot C_j\leq 0$ 
(resp.~$<0$) for every 
$j$ by 
(a). Since $\left(C_i\cdot C_j\right)$ is negative definite and 
$C_i\cdot C_j\geq 0$ for 
$i\ne j$, $a_i\geq 0$ (resp.~$>0$) holds for every $i$. 
\end{proof}

\section{Proof of Theorem \ref{f-thm1.2}}

In this section, we prove Theorem \ref{f-thm1.2}. 

\begin{defn}\label{f-def3.1}
Let $f\colon X\to Y$ be a projective surjective morphism from a smooth 
$n$-dimensional 
quasi-projective variety $X$ onto a normal quasi-projective variety $Y$. 
Let $H$ be a very ample Cartier divisor on $Y$ and let 
$A$ be a very ample Cartier divisor on $X$. 
Let $E$ be an $f$-exceptional prime divisor on $X$ with 
$\dim f(E)=e$. 
We put 
$$
C:=E\cap f^*H_1\cap \cdots \cap f^*H_e\cap 
A_1\cap \cdots \cap A_{n-e-2}, 
$$ 
where $H_i$ is a general member of $|H|$ for 
every $i$ and $A_j$ is a general member of $|A|$ for 
every $j$, and call $C$ a {\em{general curve}} 
associated to $E$, $f\colon X\to Y$, $H$, and 
$A$. We sometimes simply say that 
$C$ is a general curve associated to $E$. 
By construction, $E\cdot C<0$ and $P\cdot C\geq 0$ for 
every prime divisor $P$ on $X$ with $P\ne E$. 
We note that $C$ may be reducible and disconnected. 
We also note that 
$$
f^*H_1\cap \cdots \cap f^*H_e\cap 
A_1\cap \cdots \cap A_{n-e-2}
$$ is a smooth surface when the characteristic 
of the base field $k$ is zero by Bertini's theorem. 
Unfortunately, however, it may be singular in general. 
\end{defn}

For the proof of Theorem 
\ref{f-thm1.2}, 
we prepare several lemmas, which are easy applications of Theorem 
\ref{f-thm1.1}. 

\begin{lem}\label{f-lem3.2}
Let $f\colon X\to Y$, $H$, $A$ be as in Definition \ref{f-def3.1}. 
Let $E_1\ldots, E_m$ be $f$-exceptional prime divisors 
on $X$ such that 
$E_i\ne E_j$ for $i\ne j$ and $\dim f(E_i)=e$ for every $i$. 
Let $C_i$ be a general curve associated to $E_i$, $f\colon X\to Y$, 
$H$, and $A$ for every $i$. 
Then the matrix $\left( E_i \cdot C_j\right)$ is negative 
definite and that 
$E_i\cdot C_j\geq 0$ 
for $i\ne j$. Hence there exists an effective divisor $E$ on $X$ such that 
$\Supp E=\sum _{i=1}^m E_i$ and that 
$E\cdot C_i<0$ for every $i$.  
\end{lem}

\begin{proof}
We use the same notation as in Definition \ref{f-def3.1}. 
By restricting everything to 
$$
f^*H_1\cap \cdots \cap f^*H_e\cap 
A_1\cap \cdots \cap A_{n-e-2}, 
$$ 
we may assume that $X$ is a pure two-dimensional quasi-projective 
scheme over $k$, 
$E_i=C_i$ for every $i$, and $E_i\cdot C_j\geq 0$ for $i\ne j$ 
(see Definition \ref{f-def3.1}). 
We first treat the case where the characteristic of the base field $k$ is zero. 
In this case, $X$ is a smooth surface by Bertini's theorem. 
Then it is easy to see 
that $\left(E_i\cdot C_j\right)$ is negative 
definite (see Theorem \ref{f-thm1.1}). 
When the characteristic of the base field $k$ is positive, 
$X$ may be singular. 
In this case, by pulling everything back to a desingularization of 
$$
\left(f^*H_1\cap \cdots \cap f^*H_e\cap 
A_1\cap \cdots \cap A_{n-e-2}\right)_{\mathrm{red}}, 
$$ we can check the negative definiteness of 
$\left(E_i\cdot C_j\right)$ (see Theorem \ref{f-thm1.1}). 
Hence we can always take an effective divisor $E$ with the desired 
properties (see Lemma \ref{f-lem2.4}).  
\end{proof}

\begin{lem}\label{f-lem3.3}
Let $f\colon X\to Y$, $H$, and $A$ be as in Definition \ref{f-def3.1}. 
Let $E_1, \ldots, E_m$ be $f$-exceptional 
prime divisors on $X$ such that 
$E_i\ne E_j$ for $i\ne j$. 
Let $C_i$ be a general curve associated to $E_i$, 
$f\colon X\to Y$, $H$, and $A$ for 
every $i$. 
Then there is an effective divisor $E$ on $X$ such that 
$\Supp E=\sum _{i=1}^m E_i$ and 
that $E\cdot C_i<0$ for every $i$. 
\end{lem}

\begin{proof}
By Lemma \ref{f-lem3.2}, we can find 
an effective divisor $E^j$ on $X$ such that 
$$\Supp E^j=\sum _{\dim f(E_i)=j}E_i$$ and that 
$E^j\cdot C_i<0$ where 
$C_i$ is a general curve associated to $E_i$ with 
$\dim f(E_i)=j$. 
We note that 
$E^j\cdot C_i=0$ (resp.~$\geq 0$) 
when $C_i$ is a general curve associated to $E_i$ with $\dim f(E_i)>j$ 
(resp.~$<j$) by construction. 
We put 
$$
E=\sum _{j=0}^{n-2} m_j E^j
$$ 
with 
$$
m_0\gg m_1\gg \cdots \gg m_{n-2}>0. 
$$ 
Then $E$ is an effective divisor on $X$ with the desired 
properties. 
\end{proof}

\begin{lem}\label{f-lem3.4}
Let $f\colon X\to Y$, $H$, and $A$ be as in Definition \ref{f-def3.1}. 
Let $D$ be an $\mathbb R$-divisor on $X$ such that 
$\Supp D^-$ is $f$-exceptional and that $D\cdot C\leq 0$ for any 
general curve $C$ associated to any $f$-exceptional 
divisor on $X$. 
Then $D$ is effective. 
\end{lem}
\begin{proof}
Let $\mathcal E=\{E_1, \ldots, E_m\}$ be the set of all $f$-exceptional divisors 
on $X$. 
By pulling everything back to a desingularization $V$ 
of 
$$\left(f^*H_1\cap \cdots \cap f^*H_{n-2}\right)_{\mathrm{red}}$$ and 
using Lemma \ref{f-lem2.4} on $V$, 
we obtain that the pull-back of $D$ to $V$ is effective. 
This means that the coefficient of $E_i$ in $D$ is nonnegative when 
$\dim f(E_i)=n-2$. Assume that the coefficient of 
$E_i$ in $D$ is nonnegative when $\dim f(E_i)\geq e+1$. 
Then we pull everything back to a desingularization of 
$$\left(f^*H_1\cap \cdots \cap f^*H_e\cap 
A_1\cap \cdots \cap A_{n-e-2}\right)_{\mathrm{red}}$$ 
and use Lemma \ref{f-lem2.4} again. 
Then we obtain that the coefficient of $E_i$ 
in $D$ is nonnegative when $\dim f(E_i)\geq e$. 
We repeat this process finitely many times. 
Then we finally obtain that $D$ is effective.  
\end{proof}

\begin{lem}\label{f-lem3.5}
Let $f\colon X\to Y$ be a projective surjective morphism 
from a smooth variety $X$ onto a normal 
variety $Y$. Let $D$ be an $\mathbb R$-divisor 
on $X$. 
Then there exists an effective $f$-exceptional 
divisor 
$E$ on $X$ such that 
if $G$ is any $\mathbb R$-divisor 
on $X$ and $U$ is any Zariski open subset of $Y$ with 
the following properties:  
\begin{itemize}
\item[(a)] $G|_{f^{-1}(U)}\equiv _U D|_{f^{-1}(U)}$, 
that is, $G|_{f^{-1}(U)}$ is relatively numerically equivalent 
to $D|_{f^{-1}(U)}$ over $U$, and 
\item[(b)] the support of 
$G^-|_{f^{-1}(U)}$ is $f$-exceptional,  
\end{itemize}
then $(G+E)|_{f^{-1}(U)}$ is effective, equivalently, 
$G^-|_{f^{-1}(U)}\leq E|_{f^{-1}(U)}$ holds. 
\end{lem}

\begin{proof}
In Step \ref{f-step3.5.1}, we will treat the case where $Y$ is affine. 
In Step \ref{f-step3.5.2}, we will treat the general case. 
\setcounter{step}{0}
\begin{step}\label{f-step3.5.1}
In this step, we will construct a desired divisor $E$ under 
the extra assumption that 
$Y$ is affine. 

When $Y$ is affine, we can take an effective 
$f$-exceptional divisor $E$ on $X$ such that 
$E\cdot C<0$ for any general curve associated 
to any $f$-exceptional divisor on $X$ by Lemma \ref{f-lem3.3}. 
By replacing $E$ with $mE$ for some 
positive integer $m$, we may further assume 
that 
$(D+E)\cdot C\leq 0$ for 
any general curve $C$ associated to 
any $f$-exceptional divisor on $X$. 
Let $U$ be any Zariski open 
subset of $Y$ and let 
$G$ be any $\mathbb R$-divisor 
on $X$ satisfying (a) and (b). 
Then, 
$$
(G+E)|_{f^{-1}(U)}\cdot C=(D+E)|_{f^{-1}(U)}\cdot C\leq 0
$$ 
holds for any general curve $C$ associated to any $f$-exceptional 
divisor on $X$ by (a) and the construction of $E$, and 
the support of $(G+E)^-|_{f^{-1}(U)}$ is $f$-exceptional by (b). 
Hence, by Lemma \ref{f-lem3.4}, $(G+E)|_{f^{-1}(U)}$ is effective. 
\end{step}

\begin{step}\label{f-step3.5.2}
In this step, we will treat the general case. 

We take a finite affine Zariski 
open cover $$Y=\bigcup _{\alpha} U_\alpha$$ of $Y$. 
We consider $f\colon f^{-1}(U_{\alpha})\to U_{\alpha}$ for every $\alpha$. 
By Step \ref{f-step3.5.1}, 
we have an effective $f$-exceptional divisor $E_\alpha$ on $f^{-1}(U_\alpha)$ 
with the desired property for every $\alpha$. 
Let $E$ be an effective $f$-exceptional divisor 
on $X$ such that $E_\alpha\leq E|_{f^{-1}(U_\alpha)}$ holds for 
every $\alpha$. 
Then $E$ obviously satisfies the desired property. 
\end{step} 
We complete the proof. 
\end{proof}

Let us prove Theorem \ref{f-thm1.2}. 

\begin{proof}[Proof of Theorem \ref{f-thm1.2}]
We prove this theorem in the following two steps. 
\setcounter{step}{0}
\begin{step}\label{f-step1.2.1}
Let $E$ be any $f$-exceptional 
divisor on $X$. 
Then 
$$
\left(f_*\mathcal O_X(\lfloor tD\rfloor) \right)^{**}= 
\left(f_*\mathcal O_X(\lfloor t(D+E)\rfloor)\right)^{**} 
$$ 
holds. 
Therefore, the inclusion 
$$
f_*\mathcal O_X(\lfloor t(D+E)\rfloor)\subset 
\left(f_*\mathcal O_X(\lfloor tD\rfloor) \right)^{**}
$$ 
always holds. 
\end{step} 
\begin{step}\label{f-step1.2.2}
Let $E$ be an effective $f$-exceptional divisor 
on $X$ satisfying the property in Lemma \ref{f-lem3.5}. 
Let $\Sigma$ denote the smallest Zariski closed subset of $Y$ such 
that $f$ is equidimensional over $Y\setminus \Sigma$. 
Note that $\codim _Y\! \Sigma\geq 2$. 
By \cite[Corollary 1.7]{hartshorne}, 
$f_*\mathcal O_X(\lfloor tD\rfloor)|_{Y\setminus \Sigma}$ 
is a reflexive sheaf on $Y\setminus \Sigma$. 
Let $\mathcal E=\{E_1, \ldots, E_m\}$ be the set of 
all $f$-exceptional divisors on $X$. 
Hence $\codim_X\left(f^{-1}(\Sigma)\setminus 
\sum _{i=1}^m E_i\right)\geq 2$ holds by definition. 
We take any affine Zariski open subset $U$ of $Y$. 
Then we have the following natural inclusion and equalities 
\begin{equation}\label{f-eq1}
\begin{split}
\Gamma \left(U, \left(f_*\mathcal O_X(\lfloor tD\rfloor)\right)^{**}\right)
&\subset  
\Gamma \left(U\setminus \Sigma, \left(f_*\mathcal O_X
(\lfloor tD\rfloor)\right)^{**}\right)
\\&= 
\Gamma \left(U\setminus \Sigma, f_*\mathcal O_X(\lfloor tD\rfloor)\right)
\\ &= 
\Gamma \left(f^{-1}(U)\setminus f^{-1}(\Sigma), 
\mathcal O_X(\lfloor tD\rfloor)\right)
\\ &= 
\Gamma \left(f^{-1}(U)\setminus \sum _{i=1}^m
E_i, 
\mathcal O_X(\lfloor tD\rfloor)\right), 
\end{split}
\end{equation} 
where the inclusion follows from the torsion-freeness of 
$\left(f_*\mathcal O_X(\lfloor tD\rfloor)\right)^{**}$, the first equality 
is due to the fact that 
$f_*\mathcal O_X(\lfloor tD\rfloor)|_{Y\setminus \Sigma}$ 
is a reflexive sheaf on $Y\setminus \Sigma$, the second 
equality is obvious, and  
the final equality follows from 
$\codim _X\left(f^{-1}(\Sigma)\setminus 
\sum _{i=1}^m E_i\right)\geq 2$. 
We note that 
\begin{equation*}
\begin{split}
&\Gamma \left(f^{-1}(U)\setminus \sum _{i=1}^m
E_i, 
\mathcal O_X(\lfloor tD\rfloor)\right)
=\{\phi\in k(X)\,|\, \left((\phi)+\lfloor tD\rfloor\right)|_{f^{-1}(U)
\setminus \sum E_i}\geq 0\}\cup \{0\}, 
\end{split} 
\end{equation*} 
where $k(X)$ stands for the rational function field of $X$ and 
$(\phi)$ is the divisor associated to $\phi\in k(X)$. 
By the definition of $E$, 
if 
$$\left((\phi)+\lfloor tD\rfloor\right)|_{f^{-1}(U)
\setminus \sum E_i}\geq 0
$$ 
holds, 
then 
$$\left((\phi)+ t(D+E)\right)|_{f^{-1}(U)}\geq 0
$$ 
holds. 
Therefore, by taking the round-down, 
we obtain that 
$$\left((\phi)+ \lfloor t(D+E)\rfloor\right)|_{f^{-1}(U)}\geq 0. 
$$ 
This implies that 
\begin{equation}\label{f-eq2}
\begin{split}
&\Gamma \left(f^{-1}(U)\setminus \sum _{i=1}^m
E_i, 
\mathcal O_X(\lfloor tD\rfloor)\right)
\\&\subset 
\{\phi\in k(X)\,|\, \left((\phi)+\lfloor t(D+E)\rfloor\right)|_{f^{-1}(U)}
\geq 0\}\cup \{0\}\\ 
&= \Gamma \left(f^{-1}(U), \mathcal O_X(\lfloor t(D+E)\rfloor)\right)
\\ &= 
\Gamma \left(U, f_*\mathcal O_X(\lfloor t(D+E)\rfloor)\right). 
\end{split}
\end{equation}
Hence we get the following 
inclusion 
$$
\Gamma \left(U, \left(f_*\mathcal O_X(\lfloor tD\rfloor)\right)^{**}\right)
\subset 
\Gamma \left(U, f_*\mathcal O_X(\lfloor t(D+E)\rfloor)\right) 
$$ 
by \eqref{f-eq1} and \eqref{f-eq2}. 
This means that the opposite inclusion 
$$
\left(f_*\mathcal O_X(\lfloor tD\rfloor)\right)^{**}
\subset f_*\mathcal O_X(\lfloor t(D+E)\rfloor)
$$ 
holds. 
\end{step} 
By combining Step \ref{f-step1.2.1} with Step \ref{f-step1.2.2}, 
we see that 
the effective $f$-exceptional divisor $E$ on $X$ with the property in Lemma 
\ref{f-lem3.5} is a desired one. 
\end{proof}

Finally, we note the following statement, which is similar to 
\cite[Chapter III, 5.10.~Lemma (4)]{nakayama}. 

\begin{prop}\label{f-prop3.6} 
Let $f\colon X\to Y$ be a projective 
surjective morphism from a smooth quasi-projective 
variety $X$ onto a normal quasi-projective variety $Y$. 
Let $D$ be an $\mathbb R$-divisor on $X$. 
Assume that $D\cdot C\leq 0$ for any general 
curve $C$ associated to any $f$-exceptional 
divisor on $X$. 
Then $f_*\mathcal O_X(\lfloor D\rfloor)$ is reflexive. 
\end{prop}

\begin{proof} 
If we further assume that $Y$ is quasi-projective and $D\cdot C\leq 0$ for 
any general curve $C$ associated to any $f$-exceptional divisor on 
$X$ in Lemma \ref{f-lem3.5}, then we see that 
$G|_{f^{-1}(U)}$ is effective by Lemma \ref{f-lem3.4} and 
the proof of Lemma \ref{f-lem3.5}. 
Hence, the argument in Step \ref{f-step1.2.2} in the 
proof of Theorem \ref{f-thm1.2} works with $E=0$ 
and $t=1$. 
Therefore, we obtain 
$$
f_*\mathcal O_X(\lfloor D\rfloor)=\left(f_*\mathcal O_X(\lfloor 
D\rfloor)\right)^{**}. 
$$ 
This means that $f_*\mathcal O_X(\lfloor D\rfloor)$ is reflexive. 
\end{proof}

%%%%%%%%%%%%%%%

\end{document}